\documentclass[11pt,oneside,leqno]{amsart}

\usepackage{amsmath}
\usepackage{amsfonts}
\usepackage{amsthm}

\usepackage{color} 
\usepackage{graphicx}
\usepackage{hyperref}
\usepackage{anysize}
\marginsize{2cm}{2cm}{2cm}{2cm}

\newtheorem{theorem}{Theorem}
\newtheorem{proposition}{Proposition}
\newtheorem{lemma}{Lemma}
\newtheorem{remark}{Remark}

\begin{document}

\title[Unique continuation in water-waves]{Controllability for a non-local formulation of surface gravity waves}

\author[M. A. Fontelos]{M. A. Fontelos}
\address[M. A. Fontelos]{ICMAT-CSIC, C/Nicol\'as Cabrera, no. 13--15 Campus de Cantoblanco, UAM, 28049 Madrid, Spain.}
\email{marco.fontelos@icmat.es}

\author[R. Lecaros]{R. Lecaros}
\address[R. Lecaros]{Departamento de Matem\'atica, Universidad T\'ecnica Federico Santa Mar\'ia,  Santiago, Chile.}
\email{rodrigo.lecaros@usm.cl}

\author[J. L\'opez-R\'ios]{J. L\'opez-R\'ios}
\address[J. L\'opez-R\'ios]{(Corresponding Author) Universidad Industrial de Santander, Escuela de Matemáticas, A.A. 678, Bucaramanga, Colombia}
\email{jclopezr@uis.edu.co}

\author[A. P\'erez]{A. P\'erez}
\address[A. P\'erez]{Departamento de Matem\'atica, Universidad del B\'io-B\'io, Concepci\'on, Chile.}
\email{aaperez@ubiobio.cl}


\thanks{R. Lecaros, J. L\'opez-R\'ios and A. P\'erez have been partially supported by the Math-Amsud project
	CIPIF 22-MATH-01. R. Lecaros was partially supported by FONDECYT (Chile) Grant 1221892. J. L\'opez-R\'ios acknowledges support by Vicerrectoría de Investigación y Extensión of Universidad Industrial de Santander, project 3752.
	M. A. Fontelos was supported by projects TED2021-131530B-I00 and PID2020-113596GB-I0.
}

\maketitle

\begin{abstract}
In this paper, we study the approximate controllability of a system governed by an evolution problem known as the sloshing problem. This problem involves a spatial, nonlocal differential operator inherent in the dynamics of a two-dimensional, incompressible, non-viscous fluid within a confined domain. Our work establishes unique continuation results that enable the application of source control localized in an interior domain, allowing the aforementioned controllability.
\end{abstract}

\section{Introduction}

In this work, we are concerned with a linearized, two-dimen-sional non-local version of the water-waves system (see \cite{fontelos2023controllability,lannes2013water}) in the context of the sloshing problem and the control of the oscillations of a free surface liquid on bounded domains. We consider the initial value problem
\begin{equation}\label{I1}
	\begin{cases}
		\phi_{tt}+\mathcal{A}\phi=f, \quad &(t,x)\in(0,\infty)\times(-1,1), \\
		\phi(0,x)=\phi_0(x), &x\in(-1,1), \\
		\phi_t(0,x)=\phi_1(x), &x\in(-1,1), \\
	\end{cases}
\end{equation}
where $\phi(t,x)$ is a real-valued function and the integral, non-local operator $\mathcal{A}$, is given by (see section \ref{s2} below for details)
\begin{equation}\label{A}
	\mathcal{A}\phi:=\partial_x\left(\frac{\sqrt{1-x^2}}{\pi}P.V.\int_{-1}^{1}\frac{\phi(t,\xi)}{\sqrt{1-\xi^2}(x-\xi)}d\xi\right).
\end{equation}

The system (\ref{I1}) was first proposed in \cite{fontelos2023controllability} in the context of studying the exact controllability with an inner control acting as a source. The practical question was to control the oscillations of an inviscid, incompressible fluid-free surface on general bounded two-dimensional domains in connection with the avoidance of undesirable splashing appearing in the cooper conversion process (see \cite{brimacombe1985toward,godoy2008modeling}).

We will present in this work the existence of a general unique continuation property for the water-waves system, written here in the context of the system (\ref{I1}). Moreover, we will connect this principle with the interesting problem within the frame of hyperbolic Partial Differential Equations (PDE), namely, the approximate control property of the solution of (\ref{I1}) by source inner controls on subsets of the domain.

Let $L^2(-1,1)$ be the space of all real-valued square-integrable functions with the inner product $(\cdot,\cdot)_{L^2(-1,1)}\equiv(\cdot,\cdot)_{L^2}$, and the corresponding norm $\|\cdot\|$. Given $w(x)=:\sqrt{1-x^2}$, we are going to consider the weighted  $L^{2}$ spaces
\begin{equation*}
	L_{w}^{2}(-1,1)\equiv \left\{v:\|v\|_w<\infty\right\},
\end{equation*}
where the weighted $L^{2}_{w}$-norm (resp. $L^{2}_{w^{-1}}$) is defined by
\begin{equation*}
	\|v\|^2_{L^2_w}\equiv\|v\|^2_{w}:=\int_{-1}^{1}w\left\vert
	v\right\vert^{2}dx \ \left(\mbox{resp.} \int_{-1}^{1}w^{-1}|v|^{2}dx \right).
\end{equation*}
Moreover, we define  
\begin{equation}\label{psi1}
	H_{w^{-1}}^{1/2}(-1,1)=\left\{\psi\in L^2\; :\left\Vert\psi\right\Vert_{H_{w^{-1}}^{1/2}}^2:=\|\psi\|_{w^{-1}}^2+
	(\mathcal{A}\psi,\psi)_{L^2}<\infty \right\}.
\end{equation}
If it does not lead to confusion, we write $L_{w}^{2}$ (resp. $H_{w^{-1}}^{1/2}$) instead of $L_{w}^{2}(-1,1)$ (resp. $H_{w^{-1}}^{1/2}(-1,1)$).

Our first main result in this work, which proof is presented in Section \ref{sec:4}, is the following unique continuation property for the system (\ref{I1}).
\begin{proposition}
	\label{I_cp}  Let $T>0$, $[\phi_0,\phi_1]\in H_{w^{-1}}^{1/2}\times L^2$ and  $f:\mathbb{R}\longrightarrow \mathbb{R}$ Lipschitz, i.e. there exists $c>0$  a positive constant such that $|f(x)|\le c|x|$. 
	
	If $\phi=0$ in an open subset $M\subset (0,T)\times (-1,1)$, where $\phi\in C((0,T);H^{1/2}_{w^{-1}})$ is the solution to
	\begin{equation}\label{ac3_9}
		\begin{cases}
			\phi_{tt}+\mathcal{A}\phi=f(\phi), \quad &(t,x)\in(0,T)\times(-1,1), \\
			\phi(0,x)=\phi_0(x), &x\in(-1,1), \\
			\phi_t(0,x)=\phi_1(x), &x\in(-1,1).
		\end{cases}
	\end{equation}
	Then
	\begin{equation*}
		\phi\equiv0, \quad\text{in } (0,T)\times(-1,1).
	\end{equation*}
\end{proposition}
Unique continuation properties are useful in the context of control and inverse problems (e.g., \cite{LL2019}, \cite{LRL:2012}) and the water-waves system and some of its approximating models. Its usefulness in control theory is due to the equivalence with the approximate controllability for linear PDE (see \cite{zuazua2007controllability}), and it is involved in the classical uniqueness/compactness
approach in the proof of the stability for a PDE with localized damping. Unique continuation properties for several approximating models of the water-waves equation, as models for the propagation of one-dimensional, unidirectional waves, are proved by different methods. For instance, in the case of the KdV equation, with the aid of some Carleman estimates in \cite{escauriaza2007uniqueness,rosier2006global,saut1987unique}, by the inverse scattering approach in \cite{zhang1992unique}, and in \cite{bourgain1997compactness} by a perturbative approach and Fourier analysis. Additionally, in recent work by C. Kenig et al., \cite{kenig2020uniqueness}, several unique continuation properties are presented for the one-dimensional Benjamin-Ono and the Intermediate Long Wave equations.

The second novelty result of this paper is to apply the aforementioned principles to obtain the interior approximate controllability of the PDE (\ref{I1}). We will follow methods developed for the linear wave equation, as in \cite{fontelos2023controllability,zuazua2007controllability}, and apply them to our linear, non-local hyperbolic PDE (\ref{I1}), involving a self-adjoint operator. First, we extend the inner controllability property established in \cite{fontelos2023controllability}, to the approximate control of (\ref{I1}), by a source control $f=v\mathbf{1}_\mathcal{I}$, where $\mathbf{1}_\mathcal{I}$ denotes the characteristic function of the nonempty open subset $\mathcal{I}\subset(-1,1)$. To be more precise, we achieved the following controllability result.

\begin{theorem}\label{I_ac} Let $T>0$ and $[\phi_0,\phi_1]\in H_{w^{-1}}^{1/2}\times L^2$.  
	The system
	\begin{equation}\label{ac3_11}
		\begin{cases}
			\phi_{tt}+\mathcal{A}\phi=v\mathbf{1}_\mathcal{I}, \quad &(t,x)\in(0,T)\times(-1,1), \\
			\phi(0,x)=\phi_0(x), &x\in(-1,1), \\
			\phi_t(0,x)=\phi_1(x), &x\in(-1,1).
		\end{cases}
	\end{equation}
	is approximately controllable with a control $v\in L^2(0,T;L_w^2(-1,1))$, on a nonempty open set $\mathcal{I}\subset(-1,1)$, i.e., for any $\epsilon>0$ and $[\phi_0,\phi_1], \ [g_0,g_1]\in H_{w^{-1}}^{1/2}\times L^2$, there exists a control function $v\in L^2(0,T;L^2_w)$ such that the solution of (\ref{ac3_11}) satisfies
	\[ \|[\phi(T,\cdot),\phi_{t}(T,\cdot)]-[g_0,g_1]\|_{H_{w^{-1}}^{1/2}\times L^2}\le\epsilon. \]
\end{theorem}

Thus, the question is whether the solution of the system (\ref{ac3_11}) can be driven, from any initial state, sufficiently close to any final state in time $T$ using the action of a control with support in an open subset of $(-1,1)$. The proof of this result is presented in Section \ref{sec:4}. The reason why we have stated the above theorem as the approximate control of a linear equation with right-hand side $f=v\mathbf{1}_{\mathcal{I}}$, is because the latter term represents an injection of fluid jets through the rigid walls of the domain as shown in \cite{fontelos2023controllability}.

Controlling the surface by different methods is of practical interest in oceanography, controllability, and inverse problems theory. In this context of the two-dimensio-nal water-waves system on bounded domains, let us mention the works by Reid et al. \cite{reid1995control,reid1985boundary} where the authors addressed the null-controllability of the free surface by a source control for the linear conservation laws, both the capillarity and the non-capillarity version, respectively. Concerning nonlinear water-waves, we mention the works by Alazard \cite{alazard2017stabilization} for a two-dimensional rectangular domain, where the stabilization by an external pressure acting on a small portion of the free surface is addressed, and \cite{alazard2018control} where the local exact controllability of the two-dimensional full water-waves system by controlling the external pressure on a localized portion of the surface is studied. Moreover, in \cite{alazard2018boundary}, the author studied the boundary observability problem in a three-dimensional rectangular domain and obtained an estimate for the
energy of the system in terms of the surface velocity at the contact line with a vertical wall. About the literature concerning the control of waves by wave-makers, controllability, and stability issues in the water-waves context, we mention \cite{mottelet2000controllability,su2021strong,su2020stabilizability}. Finally, some related works using the optimal control approach are \cite{nersisyan2014generation}, where the ‘best’ moving solid bottom generating a prescribed wave is designed in the context of a BBM-type equation, and \cite{fontelos2017bottom} in the context of the general full water waves system and the inverse problem of the detection of a bottom from measurements of the free surface profile. Given the aforementioned variety of models and the diversity of the problems to be studied, it is essential to explore the extent to which the controllability problem properties established here hold for the general water-waves and non-local hyperbolic systems (refer to \cite{fontelos2017bottom}).

System (\ref{I1}) describes the evolution of the oscillations for the free boundary of gravity waves in contact with a solid container in connection with the so-called sloshing problem; namely, Dirichlet or Neumann type boundary conditions need to be imposed at the solid walls in contact with the fluid (see \cite{benjamin1979gravity,graham1983new}). Let us remark that system (\ref{I1}) shares some common properties with the wave equation due to the presence of a self-adjoint differential operator, denoted as $\mathcal{A}$ (see section \ref{sec:A:operator} for details). However, a crucial distinction arises concerning gravity waves and the eigenvalues of the Laplacian at the surface, as highlighted in \cite{fontelos2020gravity,kim2015capillary}; the eigenvalues of the problem with sidewalls behave as the square root of the eigenvalues of the problem without walls. Furthermore, the non-local nature of $\mathcal{A}$ serves as an additional motivation for the current study. Finally, it should be noted that, although the problem is posed in the half-plane, in the two-dimensional case it is easy to study the problem in any simply connected domain, through a conformal mapping, which is easily traced in the operator $\mathcal{A}$ (see \cite{fontelos2023controllability}).

Our model offers several noteworthy advantages, each rooted in careful considerations: it is a linearization of the general water-waves system, tailored specifically for scenarios where a rigid surface interfaces with the free surface. Consequently, adaptations to describe the interaction between these surfaces become necessary. The flexibility of this model allows for the accommodation of diverse two-dimensional geometries, enabled by conformal mappings between the half-plane and a simply connected domain. This adaptability proves invaluable when addressing various oceanographic problems, especially those involving different channel bottom configurations.

The rest of the paper is organized as follows. In Section \ref{s2}, we give a brief description of how the model is a one dimensional representation of the water-waves system. Section 3 is devoted to proving the well-posedness of the system \eqref{I1}.
The unique continuation property and the controllability result (Proposition \ref{I_cp} and
Theorem \ref{I_ac}, respectively) are proved in Section \ref{sec:4}.

\section{Setting of the problem}\label{s2}

In this section, we will briefly describe how the model \eqref{I1} can be interpreted as a representation of a physical system for gravity waves.

Let us consider an incompressible, non-viscous fluid inside a two-dimensional container. Let $\mathbb{R}^2_{-}=\{(x,y)\in\mathbb{R}^2:y<0\}$ be the lower half plane in two dimensions, representing the domain of the fluid. We assume side walls enclose the fluid on $y=0$ and for $x\in(-1,1)$, so there is a free boundary that obeys the conservation laws (see Figure \ref{F1}). Let $\mathbf{u}$ be the velocity field of the flow and $p$ be the hydrostatic pressure. Mathematically, on $\mathbb{R}^2_{-}$ and for the unknown $(\mathbf{u},p)$, we are in the context of Euler's equations with zero surface tension:
\begin{align*}
	\nabla\cdot\mathbf{u}&=0, \\
	\frac{\partial\mathbf{u}}{\partial t}+(\mathbf{u}\cdot\nabla)
	\mathbf{u}&=-\frac{1}{\rho}\nabla p-ge_2,
\end{align*}
where $g>0$, $-ge_2$ is the constant acceleration of gravity, $e_2$ is the unit upward vector in the vertical direction, and $\rho$ is the (constant) density of the fluid. By considering the potential function $\Phi$ of $\mathbf{u}$, so that $\mathbf{u}=\nabla\Phi$, the last system can be written as
\begin{align}  \label{pf2}
	\Delta\Phi&=0, \\
	\frac{\partial\Phi}{\partial t}+\frac{1}{2}|\nabla\Phi|^2+\frac{1}{%
		\rho }p+gy&=\text{const}.
\end{align}
This system is complemented with an impermeability condition on the walls and a kinematic condition on the free boundary (the dashed line in Figure \ref{F1}), in terms of $\Phi$ given by
\begin{align}
	\frac{\partial\Phi}{\partial n}& =0,\quad y=0, \ |x|>1 \label{pf3} \\
	\eta _{t}& =\frac{\partial\Phi}{\partial n},\quad |x|<1,  \label{pf4}
\end{align}
where $\eta(t,x)$ is the parametrization of the free surface. Note that the presence of walls laterally bounding the fluid is represented in equation (\ref{pf3}). 
Moreover, let us linearize the system and the domain as follows. Assume the domain satisfies the condition $\eta(t,x)=\epsilon\zeta(t,x)$, so that $\Phi(t,x,y)=const.+\epsilon\varphi(t,x,y)$. Then, at the first order for $\epsilon<<1$, and after re-scaling to make $\rho=g=1$, we obtain the following linearized version of \eqref{pf2}--\eqref{pf4}:
\begin{align}
	\Delta\varphi & =0, \quad (x,y)\in\mathbb{R}^2_{-},  \label{pf5} \\
	\varphi_{t}+\zeta & =0,\quad \text{at }y=0,\ |x|<1,  \label{pf6} \\
	\zeta_{t}& =\frac{\partial\varphi}{\partial n},\quad \text{at }y=0,\ |x|<1,  \label{pf7} \\
	\frac{\partial\varphi}{\partial n}&=0,\quad \text{at }y=0, \ |x|>1.
	\label{pf8}
\end{align}

\begin{figure}[tbp]
	\begin{center}
		\includegraphics[scale=1]{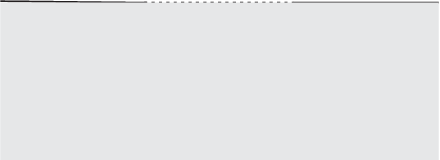}
		\put(-72,80){$1$}
		\put(-155,80){$-1$}
		\put(-125,35){$\Delta\varphi=0$}
	\end{center}
	\caption{The reference domain of the fluid enclosed by side walls.}
	\label{F1}
\end{figure}

The last system is complemented with initial conditions $\varphi_0$ and $\zeta_0$, together with the following conditions for the asymptotic behavior at infinity:
\begin{equation*}
	\partial_x\varphi, \partial_y\varphi\longrightarrow 0, \quad \text{as } y\longrightarrow-\infty \ \text{or } |x|\longrightarrow\infty.
\end{equation*}

\subsection{A non-local, integro-differential formulation of the Laplace equation}

The water-waves system above can be simplified in the following way to obtain an explicit formulation of the Laplace equation. This will help to understand the main difficulties of the problem and to formulate the approximate controllability problem in an abstract framework.

First, from \eqref{pf5}-\eqref{pf8}, one obtains the system \eqref{I1} with $\phi(t,x)\equiv\varphi(t,x,0)$ for $|x|<1$. Indeed, notice that from (\ref{pf6})-(\ref{pf7}), on the free boundary one has
\begin{equation}
	\varphi _{tt}+\frac{\partial\varphi}{\partial n}=0, \quad |x|<1.  \label{pf9}
\end{equation}
As is explained in \cite{fontelos2023controllability}, by solving the Laplace equation for $\left.\varphi_y\right|_{y=0}$, in the half-plane, we obtain
\begin{equation}\label{ac2_1}
	\left.\varphi_x(t,x,y)\right|_{y=0}=-\frac{1}{\pi }P.V.\int_{-\infty
	}^{\infty }\frac{\varphi_y(t,\xi,0)}{x-\xi }d\xi=H(\varphi_y),
\end{equation}
where $H$ is the Hilbert transform in $\mathbb{R}$ (see \cite[Chapter 5]{hochstadt2011integral}). Since the domain was flattened, we get $\displaystyle\left.\varphi_y\right|_{y=0}=\left.\frac{\partial\varphi}{\partial n}\right|_{y=0}$, and then by using (\ref{pf8}) it follows that
\begin{equation}\label{pf9_1}
	\left.\varphi_x(t,x,y)\right|_{y=0}=-\frac{1}{\pi }P.V.\int_{-1
	}^{1}\frac{\partial_n\varphi(t,\xi,0)}{x-\xi }d\xi,
\end{equation}
which is called the airfoil equation (see \cite[Chapter 5]{hochstadt2011integral}).

On the other hand, from the divergence theorem and (\ref{pf8}), the following mass conservation property 
\begin{equation}\label{mc2}
	0=\int_{\mathbb{R}_-^2}\Delta\varphi=\int_{\mathbb{R}}\frac{\partial\varphi}{\partial n}(x,0)dx=\int_{-1}^1\varphi_y(x,0)dx,
\end{equation}
holds. 

Formally computation presented in \cite[Section 3]{fontelos2023controllability}, achieved 

\begin{equation}\label{pf10_2}
	\partial_n\varphi(x)=\frac{1}{\pi}\frac{1}{\sqrt{1-x^2}}P.V.\int_{-1}^{1}\frac{\sqrt{1-\xi^2}}{x-\xi}\phi_x(\xi)d\xi,
\end{equation}

where $\phi(t,x)\equiv\varphi(t,x,0)$ for $|x|<1$ is the trace restricted to the free boundary.

\subsection{Equation analysis through an orthogonal basis}

To complement equation (\ref{pf10_2}) and establish the minimum assumptions that allow us to define the operator $\mathcal{A}$, we will consider the orthogonal system of Tchebyshev polynomials, which, thanks to the weight $w$, allows to decompose in a simple way the principal value of the integral.

Let us consider the Tchebyshev polynomials $T_n(x)$, $U_n(x)$, with $n\in\mathbb{N}_{0}:=\mathbb{N}\cup\{0\}$
of the first and second kind, respectively. Let us recall that they are defined as the polynomial solutions of the equations (see \cite{abramowitz1972} for details)
\begin{align*}
	T_n(\cos\theta)&=\cos(n\theta), \\
	U_n(\cos\theta)&=\frac{\sin((n+1)\theta)}{\sin\theta},
\end{align*}
$\theta\in[0,\pi]$. They satisfy the orthogonality relations in $L^2(-1,1)$, with the weighted inner products 
\begin{equation*}
	(f(x),g(x))_{w^{-1}}:=\int_{-1}^1\frac{f(x)g(x)}{\sqrt{1-x^2}}dx, \quad (f(x),g(x))_{w}:=\int_{-1}^1f(x)g(x)\sqrt{1-x^2}dx,
\end{equation*}
in correspondence with $\|\cdot\|_{w^{-1}}$, $\|\cdot\|_w$, in such a way that
\begin{equation}\label{oT}
	(T_{n}(x),T_{m}(x))_{w^{-1}}=\int_{-1}^1T_n(x)T_m(x)\frac{dx}{\sqrt{1-x^2}}=
	\begin{cases}
		0, \ \text{if }n\ne m, \\
		\pi, \ \text{if }n=m=0, \\
		\frac{\pi}{2}, \ \text{if }n=m\ne 0,%
	\end{cases}%
\end{equation}
\begin{equation}\label{oU}
	(U_n(x),U_m(x))_{w}=\int_{-1}^1U_n(x)U_m(x)\sqrt{1-x^2}dx=
	\begin{cases}
		0, \ \text{if }n\ne m, \\
		\frac{\pi}{2}, \ \text{if }n=m.%
	\end{cases}
\end{equation}
Moreover, for $n\ge1$, we also note that $T_{n}$ and $U_{n}$ 
verify the identities
\begin{align}
	\frac{1}{\pi}P.V.\int_{-1}^{1}\frac{T_{n}(\xi)}{\sqrt{1-\xi^{2}}(x-\xi)}d\xi&=-U_{n-1}(x),   \notag\\
	\frac{1}{\pi}P.V.\int_{-1}^{1}\sqrt{1-\xi^{2}}\frac{U_{n-1}(\xi)}{x-\xi}d\xi&=T_{n}(x),  \label{a-3}\\
	\frac{d}{dx}\left(\sqrt{1-x^2}U_{n-1}(x)\right)&=-n\frac{T_{n}(x)}{\sqrt{1-x^2}},  \notag\\
	\frac{d}{dx}T(x)&=U_{n-1}(x).\notag
\end{align}
\begin{remark}
	It is noteworthy that the sets $\{T_{n}\mid n\in \mathbb{N}_{0}\}$ and $\{U_{n}\mid n\in \mathbb{N}_{0} \}$ are basis of the spaces $L^{2}_{w^{-1}}$ and $L^{2}_{w}$, respectively.
\end{remark}
\begin{lemma}\label{rl10}
	Let $\phi\in H^{1/2}_{w^{-1}}(-1,1)$. Then, the following relation holds
	\begin{align}
		\frac{1}{\sqrt{1-x^2}}\frac{1}{\pi }P.V.\int_{-1}^{1}\sqrt{1-\xi ^{2}}%
		\frac{\phi_\xi(\xi)}{x-\xi }d\xi
		=\partial_x\left( \sqrt{1-x^2}\frac{1}{\pi }%
		P.V.\int_{-1}^{1}\frac{1}{\sqrt{1-\xi ^{2}}}\frac{\phi(\xi)%
		}{x-\xi }d\xi \right) .  \label{ex3}
	\end{align}
\end{lemma}

\begin{proof}
	Since $H^{1/2}_{w^{-1}}(-1,1)\subset L^{2}_{w^{-1}}(-1,1)$, we write
	\begin{equation}  
		\phi(x)=\sum_{n=0}^{\infty}a_{n}T_{n}(x).
	\end{equation}
	
	Then, by using the identities from \eqref{a-3}, we get for the right-hand side of \eqref{ex3}
	\begin{equation}  \label{aphi}
		\partial_x\left(\sqrt{1-x^2}\frac{1}{\pi}P.V.\int_{-1}^{1}\frac{\phi(\xi)}{ \sqrt{1-\xi^{2}}(x-\xi)}d\xi\right) =\sum_{n=1}^{\infty}na_{n} \frac{T_{n}(x)}{\sqrt{1-x^2}}.
	\end{equation}
	
	On the other hand, we note that
	\[ \phi_x(x)=\sum_{n=0}^{\infty}a_{n}T'_{n}(x)=\sum_{n=1}^{\infty}na_{n}U_{n-1}(x), \]
	which, from the identities in (\ref{a-3}), implies that the left-hand side of \eqref{ex3} can be rewritten as
	\begin{align*}
		\frac{1}{\pi\sqrt{1-x^2}}P.V.\int_{-1}^{1}\sqrt{1-\xi ^{2}}
		\frac{\phi_\xi(\xi)}{x-\xi }d\xi&=\sum_{n=1}^{\infty}na_{n}\frac{1}{\pi\sqrt{1-x^2}}P.V.\int_{-1}^{1}\sqrt{1-\xi^{2}}\frac{U_{n-1}(\xi)}{x-\xi}d\xi \\
		&=\sum_{n=1}^{\infty}na_{n} \frac{T_{n}(x)}{\sqrt{1-x^2}},
	\end{align*}
	and the proof follows.
\end{proof}
In summary, combining \eqref{pf9}, \eqref{pf10_2} and Lemma \ref{rl10}, we have the following non-local, linear gravity waves equation in two-dimensional connected domains, as the non-local Cauchy problem,
\begin{equation}\label{ac3_1}
	\begin{cases}
		\phi_{tt}+\mathcal{A}\phi=0, \quad &(t,x)\in(0,\infty)\times(-1,1), \\
		\phi(0,x)=\phi_0(x), &x\in(-1,1), \\
		\phi_t(0,x)=\phi_1(x), &x\in(-1,1),
	\end{cases}
\end{equation}
where $\displaystyle\mathcal{A}\phi(t,x):=\partial_x\left(\frac{\sqrt{1-x^2}}{\pi}P.V.\int_{-1}^{1}\frac{\phi(t,\xi)}{\sqrt{1-\xi^2}(x-\xi)}d\xi\right)$. Therefore, our main task is to prove that system (\ref{ac3_11}) is approximately controllable, as stated in Theorem \ref{I_ac}; which means that for every initial data $[\phi_0,\phi_1]$ in $H_{w^{-1}}^{1/2}\times L^{2}$, the set of reachable states $R(T;[\phi_0,\phi_1])=\left\{[\phi(T),\phi_t(T)]: v\in L^2(0,T;L_w^2)\right\}$ is dense in $H_{w^{-1}}^{1/2}\times L^2$.

\section{Well-posedness of the Cauchy problem}\label{s3}
This section serves two primary purposes. Firstly, it provides a concise overview of the well-posedness of a semi-linear system related to \eqref{ac3_1}. Readers are encouraged to refer to the detailed discussion in \cite{fontelos2023controllability} for an in-depth and comprehensive examination. Secondly, establish the necessary framework to prove the main results presented in the introduction section.

\subsection{Properties of the operator $\mathcal{A}$}\label{sec:A:operator}
In this subsection, we explore the fundamental properties of the operator $\mathcal{A}$ within the Hilbert space $H_{w^{-1}}^{1/2}(-1,1)$ with the norm given in \eqref{psi1}.
\begin{lemma}\label{lem:adjoint} Let $\phi,\psi\in H_{w^{-1}}^{1/2}(-1,1)$. Then $$(\mathcal{A}\phi,\psi)_{L^{2}}=(\phi,\mathcal{A}\psi)_{L^{2}},$$
	i.e., $\mathcal{A}$ is self-adjoint in $H_{w^{-1}}^{1/2}(-1,1)$, for the inner product in $L^2$.
\end{lemma}
\begin{proof}
	Proceeding as in the proof of Lemma \ref{rl10}, given
	\begin{equation}  \label{a-1.2}
		\phi(x)=\sum_{n=0}^{\infty}a_{n}T_{n}(x), \quad \psi(x)=\sum_{n=0}^\infty b_nT_n(x),
	\end{equation}
	we note that from \eqref{aphi} we have $\displaystyle \mathcal{A}\phi(x)= \sum_{n=1}^{\infty}na_{n} \frac{T_{n}(x)}{\sqrt{1-x^2}}$. Then, using the orthogonality relations from \eqref{oT}, we get
	\begin{equation}  \label{a1}
		(\mathcal{A}\phi,\psi)_{L^{2}}=\int_{-1}^{1}
		\left(\sum_{n=1}^{\infty}na_{n}\frac{T_{n}(x)}{\sqrt{1-x^2}}\right)\left(
		\sum_{m=0}^{\infty}b_{m}T_{m}(x)\right)dx=\frac{\pi}{2}\sum_{n=1}^{\infty }na_{n}b_{n}.
	\end{equation}
	The roles of $\psi$ and $\phi$ in the above computations can be interchanged hence the result follows.    
\end{proof}

The following Lemma, our version of the Poincar\'e inequality, implies the coerciveness of $\mathcal{A}$.
\begin{lemma}
	Let $\phi\in H_{w^{-1}}^{1/2}(-1,1)$. Then
	\begin{equation}\label{ine:A}
		\|\phi\|_{w^{-1}}^2\leq\frac{1}{4}\left(\pi^2-4\right)(\mathcal{A}\phi,\phi)_{L^{2}}+2\pi|\Bar{\phi}|^2,
	\end{equation}
	where  $\displaystyle\Bar{\phi}:=\frac{1}{2}\int_{-1}^1\phi(x)dx$.
\end{lemma}
\begin{proof}
	We write 
	\begin{equation}\label{eq:expansion}
		\phi(x):=\sum_{n=0}^{\infty}a_{n}T_{n}(x).
	\end{equation}
	Then, integrating over the interval $(-1,1)$ we have
	\begin{equation*}
		a_0\int_{-1}^1T_0(x)dx+\sum_{n=1}^{\infty}a_n\int_{-1}^1T_n(x)dx=2\Bar{\phi},
	\end{equation*}
	Since $\displaystyle\int_{-1}^1T_0(x)dx=2$, thanks to
	the Young and Cauchy-Schwarz inequalities, we get
	\begin{equation}\label{jc10}
		a_{0}^{2}\le \frac{1}{2}\sum_{n=1}^{\infty}\left(\int_{-1}^1T_n(x)dx\right)^2\sum_{n=1}^{\infty}a_n^2+2|\Bar{\phi}|^2.
	\end{equation}
	
	On the other hand, we notice that the following expansion holds
	\begin{equation}\label{eq:square:T}
		\sqrt{1-x^2}=\sum_{n=0}^{\infty}c_{n}T_{n}(x),
	\end{equation}
	since  $\sqrt{1-x^{2}}\in L^{2}_{w^{-1}}(-1,1)$. Moreover, by using the orthogonality properties \eqref{oT}, the coefficients $\{c_{n}\}_{n=0}^{\infty}$ can be computed as
	\begin{equation*}
		c_0=\frac{1}{\pi}\int_{-1}^1T_0(x)dx=\frac{2}{\pi}, \quad c_n=\frac{2}{\pi}\int_{-1}^1T_n(x)dx.
	\end{equation*}
	Then, using \eqref{eq:square:T} and \eqref{oT} we can compute 
	\begin{equation*}
		\frac{\pi}{2}=\int_{-1}^1(1-x^2)\frac{1}{\sqrt{1-x^2}}dx=\int_{-1}^1\sum_{n=0}^{\infty}c_{n}T_{n}\sum_{m=0}^{\infty}c_{m}T_{m}\frac{1}{\sqrt{1-x^2}}dx=\pi c_0^2+\frac{\pi}{2}\sum_{n=1}^{\infty}c_n^2.
	\end{equation*}
	Thus, $\displaystyle\sum_{n=1}^{\infty}\left(\int_{-1}^1T_n(x)dx\right)^2=\frac{\pi^2}{4}\sum_{n=1}^{\infty}c_n^2=\frac{\pi^2-8}{4}$. Therefore, by replacing the last expression in (\ref{jc10}), the coefficients $\{a_{n}\}_{n\geq 0} $ verify
	\begin{equation}\label{a2}
		a_{0}^{2}\leq \frac{\pi^2-8}{8}\sum_{n=1}^{\infty}a_{n}^{2}+2|\Bar{\phi}|^2.
	\end{equation}
	Finally, using the expansion \eqref{eq:expansion} we have
	\begin{align*}
		\|\phi\|_{w^{-1}}^2&=\int_{-1}^1\frac{|\phi|^2}{\sqrt{1-x^2}}dx=\pi a_0^2+\frac{\pi}{2}\sum_{n=1}^{\infty}a_n^2, 
	\end{align*}
	then using \eqref{a2} it follows that
	\begin{align*}
		\|\phi\|_{w^{-1}}^2  &\le \frac{\pi}{8}\left(\pi^2-4\right)\sum_{n=1}^{\infty}a_n^2+2\pi|\Bar{\phi}|^2 \\
		&\le \frac{\pi}{8}\left(\pi^2-4\right)\sum_{n=1}^{\infty}na_n^2+2\pi|\Bar{\phi}|^2. 
	\end{align*}
	Hence, using equation \eqref{a1} on the first expression of the right-hand side above, the inequality becomes
	\begin{align*}
		\|\phi\|_{w^{-1}}^2  &\leq\frac{1}{4}\left(\pi^2-4\right)(\mathcal{A}\phi,\phi)_{L^{2}}+2\pi|\Bar{\phi}|^2,
	\end{align*}
	and the proof is complete.
\end{proof}
\subsection{Existence of weak solutions}
In the exploration of weak solutions for the system \eqref{ac3_1}, we focus on addressing the following problem
\begin{equation}  \label{phiu}
	\phi+\mathcal{A}\phi=v,
\end{equation}
with $v$ belonging to $L^2_w$. This problem is accompanied by a weak formulation
\begin{equation}\label{weak1}
	(\phi,\psi)_{L^2}+(\mathcal{A}\phi,\psi)_{L^{2}}=(v,\psi)_{L^{2}},
\end{equation}
to be satisfied for any $\psi$ in $H^{1/2}_{w^{-1}}(-1,1)$.

As a direct consequence of the previous lemmas, we have the existence of a unique solution for the weak formulation in $H_{w^{-1}}^{1/2}$.

\begin{lemma}\label{lem:weak:existence}
	Given $v\in L^2_{w}$, there exists a unique weak solution $\phi\in H_{w^{-1}}^{1/2}$ such that (\ref{phiu}) holds.
\end{lemma}
\begin{proof}
	Let us define the bilinear form
	\[ a(\phi,\psi) := (\phi,\psi)_{L^2} + (\mathcal{A}\phi,\psi)_{L^{2}}. \]
	The continuity of $a$ is a consequence of \eqref{a1}. Indeed, applying the Cauchy-Schwartz inequality on \eqref{a1}, we get
	\begin{equation*}
		(\mathcal{A}\psi,\phi)\leq (\mathcal{A}\psi,\psi)_{L^{2}}(\mathcal{A}\phi,\phi)_{L^{2}}.
	\end{equation*}
	The above inequality and the definition of the weight $w$ enables us to write
	\begin{equation*}
		|a(\phi,\psi)|\leq \left\| \psi\right\|_{H^{1/2}_{w^{-1}}}\left\| \phi\right\|_{H^{1/2}_{w^{-1}}}.
	\end{equation*}
	For the coercivity, we first observe that $|\overline{\phi}|^2 \le \|\phi\|_{L^2}^2/2$. Then, using the later inequality and \eqref{ine:A}, we obtain
	\begin{equation}\label{coercive}
		\|\phi\|_{L^2}^2 + (\mathcal{A}\phi,\phi)_{L^2} \ge \frac{1}{\pi}\|\phi\|_{w^{-1}}^2 + \frac{4+4\pi-\pi^2}{4\pi}(\mathcal{A}\phi,\phi)_{L^{2}} \ge \frac{1}{\pi}\left\Vert\phi\right\Vert_{H_{w^{-1}}^{1/2}}^2.
	\end{equation} 
	Therefore, the result follows from the Lax-Milgram Theorem.
\end{proof}

\begin{remark}\label{r1}
	It is worth mentioning here that we have the inclusion $H^{1/2}_{w^{-1}}\subset H^{1/2}$. Indeed, considering the orthonormal base $\{e_{n}\}_{n\geq 0}\subset H^{1/2}_{w^{-1}}$, see \cite[Equation 4.20]{fontelos2023controllability}, solution of 
	\begin{equation*}
		(\mathcal{A}e_{n},\phi)_{L^{2}}=\lambda_{n}(e_{n},\phi)_{L^{2}} \mbox{ for }\phi\in H^{1/2}_{w^{-1}},
	\end{equation*}
	and the equivalent definition of the fractional Sobolev spaces 
	\begin{equation*}
		H^{1/2}(-1,1):=\left\{\phi:\sum_{n}\lambda_{n}
		\left(\int_{-1}^{1}\frac{\phi e_{n}}{\sqrt{1-x^2}}dx\right)^{2}<\infty\right\},
	\end{equation*}
	\begin{equation*}
		H^{-1/2}(-1,1):=\left\{\phi:\sum_{n}\frac{1}{
			\lambda_{n}}\left(\int_{-1}^{1}\frac{\phi e_{n}}{\sqrt{1-x^2}}dx\right)^{2}<\infty\right\}.
	\end{equation*}
	
	Then, the following equivalence holds
	\begin{align*}
		\|\phi\|_{H^{1/2}}^2&=\sum_{n}\lambda_{n}\left(\int_{-1}^{1}\frac{\phi e_{n}}{\sqrt{1-x^2}}dx\right)^{2} \\
		&=\sum_{n}\left(\int_{-1}^{1}\frac{\phi\lambda_{n}e_{n}}{\sqrt{1-x^2}}dx\right)\left(\int_{-1}^{1}\frac{\phi e_{n}}{\sqrt{1-x^2}}dx\right) \\
		&=(\mathcal{A}\phi,\phi)_{L^{2}}.
	\end{align*}
	Thus, recalling the definition \eqref{psi1} we obtain
	\begin{align*}
		\|\phi\|_{H^{1/2}}^2&\le \|\phi\|_{H^{1/2}_{w^{-1}}}.
	\end{align*}
\end{remark}
The following result guarantees the well-posedness for the non-homogeneous system related to \eqref{I_ac}.
\begin{theorem}\label{ac3_th1}
	Let $T>0$, $[\phi_0,\phi_1]\in H_{w^{-1}}^{1/2}\times L^2$ and  $f:\mathbb{R}\longrightarrow \mathbb{R}$ Lipschitz. Then, the system
	\begin{equation}\label{non-hom}
		\begin{cases}
			\phi_{tt}+\mathcal{A}\phi=f(\phi), \quad &(t,x)\in(0,\infty)\times(-1,1), \\
			\phi(0,x)=\phi_0(x), &x\in(-1,1), \\
			\phi_t(0,x)=\phi_1(x), &x\in(-1,1),
		\end{cases}
	\end{equation}
	has a unique solution
	\begin{equation*}
		[\phi,\phi_t]\in C([0,T];H_{w^{-1}}^{1/2}\times L^2).
	\end{equation*}
\end{theorem}
\textbf{Sketch of the proof:}
We consider the Hilbert space $X:=H^{1/2}_{w^{-1}}\times L^2,$ with the scalar product 
\begin{equation*}
	((u,v),(w,z))_X:=(\mathcal{A}u,w)_{L^2}+(v,z)_{L^2}.
\end{equation*}
Defining the linear operator ${\bf A}:D({\bf A})\subset X\longrightarrow X$ given by ${\bf A}(u,v):=(v,-\mathcal{A}u)$, where
$D({\bf A}):=\left\{[u,v]\in X|\; {\bf A}(u,v) \in X \right\}$. We have that, with $\displaystyle U:=\begin{pmatrix} \phi\\ \phi_{t}\end{pmatrix}$, the homogeneous system associated to \eqref{non-hom} is equivalent to
\begin{equation}\label{system:wp}
	\begin{cases}
		U'(t)=\mathbf{A}U(t), \quad &(t,x)\in(0,\infty)\times(-1,1), \\
		U(0)=U_{0}, &x\in(-1,1),
	\end{cases}
\end{equation}
where $\displaystyle U_{0}:=\begin{pmatrix} \phi_{0}\\ \phi_{1}\end{pmatrix}$. We know that \eqref{system:wp} is well-posed if and only if the operator $\mathbf{A}$ is a generator of a contraction isometry group in $X$. Recalling the characterization given by \cite[Theorem 3.4.4]{CH:1990} we need to show that $\mathbf{A}$ is $m-$dissipative in $X$(see \cite[Definition 2.2.2]{CH:1990}). Thanks to remark \ref{r1} we notice $C_c^{\infty}([-1,1])\subset H^{1/2}_{w^{-1}}$, implying $D({\bf A})$ is dense in $X$. In turn, thanks to \cite[Proposition 2.4.2]{CH:1990}, the dissipative property is equivalent to prove $({\bf A}(u,v),(u,v))_X=0$ for all $[u,v]\in D({\bf A})$ which follows directly from the fact of  $\mathcal{A}$ being  self-adjoint (see Lemma \ref{lem:adjoint}). Indeed, we have
$$({\bf A}(u,v),(u,v))_X=((v,-\mathcal{A}u),(u,v))_X=(\mathcal{A}v,u)_{L^2}-(\mathcal{A}u,v)_{L^2}=0.$$
Therefore, ${\bf A}$ is dissipative. Now, to conclude that $\mathbf{A}$ es $m$-dissipative(see \cite[Proposition 2.2.6 ]{CH:1990}), we need to show that for any $[h,g]\in X$  there exists $[u,v]\in D(\textbf{A})$ such that 
\begin{equation}\label{eq:system}
	(u,v)-{\bf A}(u,v)=(h,g).
\end{equation}
To this end, for $[h,g]\in X$ given, we note that \eqref{eq:system} is equivalent to
$$
\left\{\begin{array}{l}
	u-v=h,\\
	v+\mathcal{A}u=g.
\end{array}\right.
$$
Thus, we have that $h+g\in L^2\subset L^2_{w}$. Then, by Lemma \ref{lem:weak:existence}, there exists a solution $u\in H^{1/2}_{w^{-1}}$, satisfying $u+\mathcal{A}u=h+g\in L^2$ implying that $v=u-h\in H^{1/2}_{w^{-1}}$. Therefore, we found $[u,v]\in D({\bf A})$ such that \eqref{eq:system} holds for all $[h,g]\in X$. In conclusion, ${\bf A}$ is $m$-dissipative in $X$. The proof of $-{\bf A}$ being $m$-dissipative is analogous. 

At this stage, we have the well-possedness of system \eqref{system:wp} for bounded time intervals. To obtain global existence of all solutions we need the dissipation of the following energy (see for instance \cite[Proposition 6.3.1]{CH:1990})
\begin{equation*}
	E(t)=\int_{-1}^1|\phi|^2 dx +\int_{-1}^1|\phi_t|^2dx+(\mathcal{A}\phi,\phi)_{L^2}.
\end{equation*}
However, by a density argument we can take the initial data in $D(\mathbf{A})$ to ensure that $\phi_{t}\in L^{2}$. In this case we have $\frac{d}{dt}|\phi|^2=\frac{1}{2}\phi\,\phi_t$, then
\begin{equation}\label{energy3}
	\frac{1}{2}\frac{d}{dt}\int_{-1}^1|\phi|^2dx\le \frac{1}{2}\int_{-1}^1|\phi|^2dx+\frac{1}{2}\int_{-1}^1|\phi_t|^2dx.
\end{equation}
On the other hand, multiplying the main equation of \eqref{non-hom} by $\phi_t$ and integrating over $x\in(-1,1)$ it follows that
\begin{equation*}
	\frac{1}{2}\frac{d}{dt}\int_{-1}^1|\phi_t|^2dx+(\mathcal{A}\phi,\phi_t)_{L^{2}}=\int_{-1}^1f(\phi)\phi_t dx.
\end{equation*}
Since $\mathcal{A}$ is self-adjoint,
\begin{equation*}
	\frac{1}{2}\frac{d}{dt}\int_{-1}^1|\phi_t|^2dx+\frac{1}{2}\frac{d}{dt}(\mathcal{A}\phi,\phi)_{L^{2}}=\int_{-1}^1f(\phi)\phi_t dx.
\end{equation*}
Now, by Young's inequality and the Lipschitz assumption on $f$,
\begin{equation}\label{energy6}
	\frac{1}{2}\frac{d}{dt}\int_{-1}^1|\phi_t|^2dx+\frac{1}{2}\frac{d}{dt}(\mathcal{A}\phi,\phi)_{L^{2}}\le\frac{c^2}{2}\int_{-1}^1|\phi|^2dx+\frac{1}{2}\int_{-1}^1|\phi_t|^2dx\le CE(t).
\end{equation}

Therefore, combining (\ref{energy3}) and (\ref{energy6}) we obtain
\begin{equation}\label{energy}
	\frac{d}{dt}E(t)\le CE(t),
\end{equation}
and thus, by Gronwall's inequality
\begin{equation}\label{energy2}
	E(t)\le E(t_0)e^{C(t-t_0)}, \ \forall t\in(0,T).
\end{equation}
\begin{remark}
	We note that it is possible to state the well-possedness of Theorem \ref{ac3_th1} for a more general right-hand side $f$. For instance, we can take $f:H_{w^{-1}}^{1/2}\longrightarrow L^{\infty}(L_w^2)$. In this case the Lipschitz condition must be replaced by $|f(t,x)|\leq c|x|$ uniformly in $t$. Moreover, whether we consider a right-hand side of the form $f+g$, with $g\in L^{2}(L^{2}_{w})$ and $f$ as stated in Theorem \ref{ac3_th1}, it is possible to follow the argumentation above to estimate the energy.
\end{remark}

\section{Proof of Proposition \ref{I_cp}  and the approximate controllability}\label{sec:4}

Using the frame stated above, namely the relation of the velocity potential in the half-plane and the free boundary, it is possible to state a continuation principle for the non-local operator $\mathcal{A}$. To this end, the following lemma connects the non-local operator $\mathcal{A}$ with the normal derivative of a local system. In what follows, we will use the ideas from \cite{hochstadt2011integral} about the airfoil equation to explain the choice of our operator $\mathcal{A}$.
\begin{lemma} Let $\phi\in H^{1/2}(-1,1)$ and $\varphi$ be such that
	\begin{equation}\label{auxiliar:01}
		\begin{cases}
			\Delta\varphi=0,& \mathbb{R}^2_{-},  \\
			\partial_n\varphi(\cdot,0)=0, & |x|>1, \\
			\varphi(x,0)=\phi(x), &|x|<1.
		\end{cases}
	\end{equation}
	Then
	\begin{equation}\label{ac3_7}
		\varphi_y(\cdot,0)=\mathcal{A}\phi,\quad \textrm{in }(-1,1),
	\end{equation}
	and $\varphi (x,0)\in H^{1/2}(\mathbb{R})$.
\end{lemma}
\begin{proof}
	We have $\bigl.\varphi_y\bigr|_{y=0}=\partial_n\varphi$. Thus, it is enough to prove that $\partial_{n}\varphi(\cdot, 0)=\mathcal{A}\phi$. To do this, given $\psi\in L^2[0,\pi]$, let us define the following operators
	\begin{equation*}
		K\psi(y):=\frac{1}{\pi}P.V.\int_0^\pi\frac{\sin(z)}{\cos(y)-\cos(z)}\psi(z)dz 
	\end{equation*}
	and
	\begin{equation*}
		K^*\psi(y):=-\frac{1}{\pi}P.V.\int_0^\pi\frac{\sin(y)}{\cos(y)-\cos(z)}\psi(y)dz.
	\end{equation*}
	We note that these operators verify (see \cite[Section 5.2, Example 9]{hochstadt2011integral}) 
	
	\begin{align} 
		K\sin(ny)&=\cos(ny), \quad n=1,2,..., \notag\\
		K^*\cos(ny)&=\sin(ny), \quad n=1,2,..., \label{eq:properties:K}\\
		K^*1&=0.\notag
	\end{align}
	
	Following \cite{fontelos2023controllability}, by solving the Laplace equation in the half-plane, we obtain
	\begin{equation}\label{ac2_1.2}
		\varphi_x(x,0)=H(\partial_{n}\varphi),
	\end{equation}
	where $H$ is the Hilbert transform in $\mathbb{R}$, and then by using $\partial_n\varphi=0$, in $|x|>1$, it follows that
	\begin{equation}\label{pf9_1_v2}
		\varphi_x(x,0)=-\frac{1}{\pi }P.V.\int_{-1
		}^{1}\frac{\partial_n\varphi(\xi,0)}{x-\xi }d\xi,
	\end{equation}
	which is called the airfoil equation (see \cite[Chapter 5]{hochstadt2011integral}).
	
	By introducing the variables
	\[  x=\cos(y), \quad \xi=\cos(z), \]
	\[ \alpha(y)=\varphi_x(\cos(y),0)\sin(y), \quad \beta(z)=\partial_n\varphi(\cos(z),0)\sin(z), \]
	it follows that the equation (\ref{pf9_1_v2}) becomes
	\begin{equation}\label{rl1}
		\alpha(y)=-\frac{1}{\pi }P.V.\int_{0}^{\pi}\frac{\sin(y)}{\cos(y)-\cos(z)}\beta(z)dz=K^*\beta(y).
	\end{equation}
	We observe $\phi\in H^{1/2}\subset H^{1/2}_{w^{-1}}$, then $\varphi_x=\phi_x\in L^2_w$, which imply $\alpha\in L^2(0,\pi)$ and by expanding
	\begin{equation}\label{rl3}
		\alpha(y)=\sum_{n=1}^\infty a_n\sqrt{\frac{2}{\pi}}\sin(ny),
	\end{equation}
	and replacing the above expression into (\ref{rl1}) we get
	\begin{equation}\label{eq:alpha}
		K^\ast\beta(y)=\sum_{n=1}^\infty a_n\sqrt{\frac{2}{\pi}}\sin(ny).
	\end{equation}
	
	Using the aforementioned properties of the operator $K^*$ in 
	\eqref{eq:alpha},
	we obtain 
	\begin{equation}\label{eq:beta}
		\beta(y)=\frac{1}{\sqrt{\pi}}a_0+\sum_{n=1}^\infty a_n\sqrt{\frac{2}{\pi}}\cos(ny),
	\end{equation}
	which imply $\beta\in L^2(0,\pi)$, then $\displaystyle\bigl.\partial_n\varphi\bigr|_{y=0}\in L^2_w$. 
	
	In turn, using the identity of $K$ from \eqref{eq:properties:K} and \eqref{eq:alpha}, the expansion \eqref{eq:beta} can be rewritten as
	\begin{align*}
		\beta(y)&=\frac{1}{\sqrt{\pi}}a_0+\sum_{n=1}^\infty a_n\sqrt{\frac{2}{\pi}}K\sin(ny) \\
		&=\frac{1}{\sqrt{\pi}}a_0+K\alpha(y) \\
		&=\frac{1}{\sqrt{\pi}}a_0+\frac{1}{\pi}P.V.\int_0^\pi\frac{\sin(z)}{\cos(y)-\cos(z)}\alpha(z)dz.
	\end{align*}
	
	Therefore, reverting to the original variables in the above identity, one obtains
	\begin{align*}
		\partial_n\varphi(x,0)\sin(y)&=\frac{1}{\sqrt{\pi}}a_0+\frac{1}{\pi}P.V.\int_1^{-1}\frac{\sqrt{1-\xi^2}}{x-\xi}\phi_x(\xi)\sin(z)\frac{d\xi}{-\sin(z)},
	\end{align*}
	namely
	\begin{equation}\label{rl6}
		\partial_n\varphi(x)=\frac{a_0}{\sqrt{\pi}}\frac{1}{\sqrt{1-x^2}}+\frac{1}{\pi}\frac{1}{\sqrt{1-x^2}}P.V.\int_{-1}^{1}\frac{\sqrt{1-\xi^2}}{x-\xi}\phi_x(\xi)d\xi.
	\end{equation}
	
	Now, from the expansion given in \eqref{eq:beta} it follows that
	\[ \int_\pi^0\beta(y)dy=-\sqrt{\pi}a_0, \]
	which is equivalent to
	\[ \int_{-1}^1\partial_n\varphi(x)dx=\sqrt{\pi}a_0, \]
	and thanks to \eqref{auxiliar:01}, from the above expression we conclude that $a_{0}=0$. Hence, from (\ref{rl6}) and the last equality, the solution of the Laplace equation on the line, such that $\varphi_y|_{(-1,1)}\in L_w^2(-1,1)$, is given by
	\begin{equation}\label{pf10}
		\partial_n\varphi(x)=\frac{1}{\pi}\frac{1}{\sqrt{1-x^2}}P.V.\int_{-1}^{1}\frac{\sqrt{1-\xi^2}}{x-\xi}\phi_x(\xi)d\xi.
	\end{equation}
	
	Therefore, according to Lemma \ref{rl10} we have that
	\[ \partial_n\varphi(\cdot,0)=\mathcal{A}\phi, \]
	and the proof of \eqref{ac3_7} is complete.
	
	Finally, to prove $\varphi (x,0)\in H^{1/2}(\mathbb{R})$, we recall that by definition
	\begin{equation}\label{eq:0}
		\left\Vert \varphi \right\Vert _{H^{1/2}(\mathbb{R})}:=\left\Vert \varphi
		\right\Vert _{L^{2}(\mathbb{R})}+\left( \int_{-\infty }^{\infty
		}\left\vert k\right\vert \left\vert \widehat{\varphi }\right\vert
		^{2}dk\right)^{\frac{1}{2}}.
	\end{equation}
	We notice that combining \eqref{ac2_1.2} and \eqref{a1}, from the second term in the right-hand side above, we have
	\begin{equation}\label{eq:1}
		\int_{-\infty }^{\infty }\left\vert k\right\vert \left\vert \widehat{\varphi
		}\right\vert ^{2}dk=\int_{-\infty }^{\infty }\varphi H\varphi
		_{x}dx=\int_{-\infty }^{\infty }\varphi \frac{\partial \varphi }{\partial n}%
		dx=\int_{-1}^{1}\varphi \frac{\partial \varphi }{\partial n}dx=\frac{\pi }{2}\sum_{n=1}^{\infty} na_{n}^{2}.
	\end{equation}
	Now, let us focus on the first term from the right-hand side of \eqref{eq:0}. By defining $\displaystyle\chi=\int_{-\infty}^x\varphi(s,y)ds$ we have
	\begin{eqnarray*}
		\left\Vert \varphi \right\Vert _{L^{2}(\mathbb{R})}^{2} &=&\int_{-\infty
		}^{\infty }\varphi ^{2}dx=-\int_{-\infty }^{\infty }\chi\varphi
		_{x}dx.
	\end{eqnarray*}
	Then, by definition of the Hilbert transform and \eqref{pf10}, it follows that
	\begin{eqnarray*}
		\left\Vert \varphi \right\Vert _{L^{2}(\mathbb{R})}^{2}&=&-\int_{-\infty }^{\infty }H\chi H\varphi _{x}dx\\
		&=&-\int_{-1}^{1}H\chi
		\frac{\partial \varphi }{\partial n}dx \\
		&=&-\int_{-1}^{1}H\chi \partial_{x}\left( \sqrt{1-x^{2}}\frac{1}{\pi }P.V.\int_{-1}^{1}%
		\frac{\phi (\xi )}{\sqrt{1-\xi ^{2}}(x-\xi )}d\xi \right) dx.
	\end{eqnarray*}
	Now, after integration by parts on the last expression and recalling the definition of $\chi$, we obtain
	\begin{eqnarray*}
		\left\Vert \varphi \right\Vert _{L^{2}(\mathbb{R})}^{2}&=&\int_{-1}^{1}H\varphi \left( \sqrt{1-x^{2}}\frac{1}{\pi }P.V.\int_{-1}^{1}%
		\frac{\phi (\xi )}{\sqrt{1-\xi ^{2}}(x-\xi )}d\xi \right) dx \\
		&\leq &\left( \int_{-\infty }^{\infty }(H\varphi )^{2}dx\right) ^{\frac{1}{2}%
		}\left( \int_{-1}^{1}\left( \sqrt{1-x^{2}}\frac{1}{\pi }P.V.\int_{-1}^{1}%
		\frac{\phi (\xi )}{\sqrt{1-\xi ^{2}}(x-\xi )}d\xi \right) ^{2}dx\right) ^{%
			\frac{1}{2}},
	\end{eqnarray*}
	where for the last term at the right-hand side above, we have used the Cauchy-Schwarz inequality. Moreover, by expanding $\displaystyle
	\phi(x)=\sum_{n=0}^\infty a_nT_n(x)$ and using the first identity from \eqref{a-3}, we get that the above inequality becomes
	\begin{eqnarray}\label{eq:22}
		\left\Vert \varphi \right\Vert _{L^{2}(\mathbb{R})}^{2}
		&\leq &\left\Vert \varphi \right\Vert_{L^{2}(\mathbb{R})}\left(
		\int_{-1}^{1}(1-x^{2})\left( \sum_{n=1}^{\infty} a_{n}U_{n-1}(x)\right) ^{2}dx\right)^{\frac{1}{2}} \notag\\
		&\leq &\left\Vert \varphi \right\Vert_{L^{2}(\mathbb{R})}\left( \int_{-1}^{1}\sqrt{1-x^{2}}\left(\sum_{n=1}^{\infty} a_{n}U_{n-1}(x)\right) ^{2}dx\right) ^{\frac{1}{2}}
		\\
		&\leq &C\left\Vert \varphi \right\Vert _{L^{2}(\mathbb{R})}\left( \sum_{n=1}^{\infty} a_{n}^{2}\right)^{\frac{1}{2}}\notag,
	\end{eqnarray}
	where we have used \eqref{oU} for the last inequality. Hence, combining \eqref{eq:1} and \eqref{eq:22} on \eqref{eq:0} we conclude that there exists a constant $C>0$ such that
	\begin{equation*}
		\left\Vert \varphi \right\Vert _{H^{1/2}(\mathbb{R})}\leq
		C\left\Vert \phi \right\Vert _{H_{w^{-1}}^{1/2}}.
	\end{equation*}
\end{proof}
The last lemma and the setting established so far are the keys to formulating the following unique continuation principle for the linearized water-waves problem with side walls. Notice that, although we are using a formulation of the problem on the real line, we take advantage of a relation with the Laplace equation inside the domain. Thus, the classical unique continuation principle for the Laplacian is transmitted to the Dirichlet-Neumann operator on the boundary.
\begin{lemma}\label{lem:unique:continuation}
	Let $\mathcal{I}\subset(-1,1)$ be an open set. If
	\[ \phi=\mathcal{A}\phi=0 \ \text{in }\mathcal{I}, \]
	then $\phi\equiv0$ in $(-1,1)$.
\end{lemma}
\begin{proof}
	The key point is the relation between the conservation laws on the free surface and the solution of an elliptic problem on the interior of the domain. That is, given $\phi\in H^{1/2}(-1,1)$, by the trace theorem, there exists $\varphi\in H^1(\mathbb{R}\times\mathbb{R}_{-})$ such that
	\begin{equation*}
		\begin{cases}
			\Delta\varphi=0,& \mathbb{R}^2_{-},  \\
			\partial_n\varphi(\cdot,0)=0, & |x|>1, \\
			\varphi=\phi, & |x|<1.
		\end{cases}
	\end{equation*}
	From (\ref{ac3_7}) and the hypothesis we have
	\[ \partial_n\varphi=\varphi_y=\mathcal{A}\phi=0, \quad\text{in }\mathcal{I}. \]
	Moreover $\varphi=0$ in $\mathcal{I}$. Therefore, by the unique continuation principle for the Laplace operator on the half-plane,
	\[ \phi\equiv0, \quad\text{in }(-1,1). \]
\end{proof}

\begin{remark}
	Concerning the unique continuation principle above and the way it was applied here. In the two-dimensional case, by following \cite{linares2022unique} one extends the analytic solution of the laplacian to a ball in $\mathbb{R}^2$ containing $\mathcal{I}$; since it contains an acummulation point where $\varphi$ is zero, one has $\varphi$ is identically zero.
	
	In the $n$-dimensional case, on the interval $\mathcal{I}$, one can extend the solution by zero to a ball. Then, by using Green functions, one concludes the extension is analytic and thus is zero inside the ball.
\end{remark}

We obtain our first main result in this work by extending the last result to the evolutionary case, as is classical for the linear wave equation for open sets in time.
\begin{proof}[\textbf{Proof of Proposition \ref{I_cp}}]
	Let $\mathcal{I}\subset(-1,1)$ an open subset and $t_0\in(0,T)$ such that $\{t_0\}\times \mathcal{I}\subset M$, where $M$ is an open subset of $(0,T)\times (-1,1)$ that verify 
	\begin{equation}\label{eq:hipothesis}
		\phi(t,x)=0,\mbox{ for all }(t,x)\in M.
	\end{equation}
	Then, we have
	\begin{equation}\label{ac3_10}
		\phi(t_0,x)=\phi_t(t_0,x)=\phi_{tt}(t_0,x)=0, \text{ for all }x\in \mathcal{I}.
	\end{equation}
	
	We note that from \eqref{eq:hipothesis} we have $\Bigl.\phi_{tt}\Bigr|_M\equiv0$, and since $\phi$ solves the main equation from the system (\ref{ac3_9}), we obtain $\mathcal{A}\phi=f(\phi)$. Moreover, using that $f$ is Lipschitz and (\ref{ac3_10}), it follows that $f(\phi)=0$ for all $(t,x)\in M$. Therefore, from Lemma \ref{lem:unique:continuation}, we obtain $\phi(t_0,\cdot)\equiv0$ on $(-1,1)$. 
	
	In turn, in the previous efforts we showed that $\Bigl.\mathcal{A}\phi\Bigr|_M\equiv0$. Noticing that $\mathcal{A}$ is an operator independent on time, we get $\Bigl.\mathcal{A}\phi_t\Bigr|_M\equiv0$.  Thus, thanks to \eqref{ac3_10}, we can apply Lemma \ref{lem:unique:continuation} to $\phi_{t}$ implying
	$\phi_t(t_0,\cdot)\equiv0$ on $(-1,1)$.
	
	Finally, from the two facts above and the uniqueness of the solution for the system \eqref{ac3_9}, $\phi\equiv C=0$ on $(0,T)\times (-1,1)$.
\end{proof}

As is customary in control theory, the above result provides us with a controllability theorem in a classical way, as follows.

\begin{proof}[\textbf{Proof of Theorem \ref{I_ac}}]
	Let $\phi$ be the solution of (\ref{ac3_11}), with initial data \break $[\phi_0,\phi_1]\in H^{1/2}_{w^{-1}}\times L^2$. Defining
	\[ R(T;[\phi_0,\phi_1])=\left\{[\phi(T),\phi_t(T)]: v\in L^2(0,T;L_w^2)\right\}, \]
	we notice that $R$ is an affine subspace of $H_{w^{-1}}^{1/2}\times L^2$. Since we are dealing with a linear problem, it is
	sufficient to consider null initial conditions $\phi_0=\phi_1=0$.\\
	We observe that the adjoint system related to the system \eqref{ac3_11} is given by 
	\begin{equation}\label{system:adjoint}
		\begin{cases}
			z_{tt}+\mathcal{A}z=0, \quad &(t,x)\in(0,\infty)\times(-1,1), \\
			z(T,x)=g_0(x), &x\in(-1,1), \\
			z_t(T,x)=-g_1(x), &x\in(-1,1), 
		\end{cases}
	\end{equation}
	where $[g_0,g_1]\in H^{1/2}_{w^{-1}}\times L^2$. It is known that the density of $R(T;[\phi_{0},\phi_{1}])$ in $H_{w^{-1}}^{1/2}\times L^{2}$ is equivalent to proving that if \begin{equation}\label{ac3_12}
		(g_0,\phi_t(T;v))_{L^2}+(g_1,\phi(T;v))_{L^2}=0, \quad\forall v\in L^2(\mathcal{I}),
	\end{equation}
	then $g_{0}=g_{1}=0$. To show this, multiply both sides of the main equation from \eqref{system:adjoint} by the solution $\phi(v)$ of (\ref{ac3_11}) and after integrating by parts, we obtain
	\[ \int\int_{(0,T)\times \mathcal{I}}z v=(g_0,\phi_t(T;v))_{L^2}+(g_1,\phi(T;v))_{L^2}. \]
	From (\ref{ac3_12}) we deduce
	\[ z=0, \quad\text{in }(0,T)\times \mathcal{I}. \]
	Thus, from Proposition \ref{I_cp} we conclude
	\[ z\equiv 0, \quad\text{in }(0,T)\times(-1,1), \]
	which implies $g_0=g_1=0$, and the theorem is proved.
\end{proof}

\section{Comments}

Concerning PDE and non-local operators, in recent years, many papers on unique continuation principles and inverse source detection problems have been written for non-local operators in the time variable, see for example \cite{lin2022classical} and the works of Yamamoto et al. \cite{kian2021uniqueness,li2021inverse,liu2021recovering}. As far as non-local spatial operators are concerned, we mention the works of G.Ponce et. al. \cite{kenig2020uniqueness,linares2022unique}, where they study unique continuation principles for the Benjamin-Ono equation and some related dispersive models. We also mention \cite{bhattacharyya2021inverse}, where a continuation principle is established for a PDE involving the Fractional Laplacian to study the inverse problem of determining the unknown coefficients from the exterior measurements of the corresponding Cauchy data. In this direction,  a possible future work could be addressing the problem of identifying the location where a fluid injection occurs on the solid walls of the domain from wave measurements inside the container, provided we know the time at which these jet injections occur  (see \cite{yamamoto1995stability}).

In the literature, several approaches have been explored to address controllability in systems featuring spatial non-local terms, see for instance \cite{FLZ-2016,BHSM-2019,FCLNN2019}. To prove approximate controllability results, classical methodology has been applied in works such as \cite{Z:2021,LR-W:2021,W:2019,W-Z:2020} for the fractional Laplace operator. However, our approach distinguishes itself by establishing a unique continuation result for a non-local system through its connection with a local system via the trace operator. 

Since the model easily accommodates the inclusion of a surface tension term and the study of capillary waves as it was considered in \cite{kim2015capillary}, one interesting extension would be to derive the corresponding operator in this case and study the controllability issues accordingly.

Finally, concerning the right-hand side term, let us comment the following. First, the derivation of this model inherently facilitates the incorporation of control mechanisms involving the injection of fluid jets through the container walls, that is $f=v\mathbf{1}_{\mathcal{I}}$ (see \cite{fontelos2023controllability}). Second, since our unique continuation result is valid for a nonlinear equation it would be interesting to consider the full Bernoulli equation with nonlinear terms written as the right hand side $f(\phi)$ as in Proposition \ref{I_cp}. Several strategies involving fixed-point theorems can be applied to investigate controllability (see for instance \cite{Zuazua:93:semilinear}). This is particularly promising since the linear system exhibits exact (\cite{fontelos2023controllability}) and approximate control properties. Third, a general right-hand side is important if one wants to study the inverse problem of spatial source recovering from measurements of the normal derivative on the boundary, by the controllability method, as in \cite{yamamoto1995stability}.

\bibliographystyle{abbrv}
\bibliography{referencias}
\end{document}